\documentclass[10pt]{amsart}

\usepackage{mystyle}

\usepackage{ccfonts,eulervm} \usepackage[T1]{fontenc}

\usepackage [english]{babel}
\usepackage [autostyle, english = american]{csquotes}
\MakeOuterQuote{"}

\begin{document}

\begin{center}
\title{Mixing actions of countable groups are almost free}
\author{Robin D. Tucker-Drob}
\end{center}

\begin{abstract}
A measure preserving action of a countably infinite group $\Gamma$ is called totally ergodic if every infinite subgroup of $\Gamma$ acts ergodically. For example, all mixing and mildly mixing actions are totally ergodic. This note shows that if an action of $\Gamma$ is totally ergodic then there exists a finite normal subgroup $N$ of $\Gamma$ such that the stabilizer of almost every point is equal to $N$. Surprisingly the proof relies on the group theoretic fact (proved by Hall and Kulatilaka as well as by Kargapolov) that every infinite locally finite group contains an infinite abelian subgroup, of which all known proofs rely on the Feit-Thompson theorem.

As a consequence we deduce a group theoretic characterization of countable groups whose non-trivial Bernoulli factors are all free: these are precisely the groups that posses no finite normal subgroup other than the trivial subgroup.
\end{abstract}
\maketitle

\section{Introduction}
Let $\Gamma$ be a countably infinite discrete group and let $\bm{a}$ be a measure preserving action of $\Gamma$, i.e., $\bm{a} = \Gamma \cc ^a (X,\mu )$ where $X$ is a standard Borel space, $\mu$ is a Borel probability measure on $X$, and $a:\Gamma\times X\ra X$ is a Borel action of $\Gamma$ on $X$ that preserves $\mu$. In this note we examine how ergodicity and mixing properties of $\bm{a}$ can influence, and be influenced by, the freeness behavior of $\bm{a}$ and its factors. When $\bm{a}$ is not ergodic for example, the ergodic decomposition of $\bm{a}$ directly exhibits a \emph{non-trivial} action (i.e., with underlying measure not a point mass) that is a factor of $\bm{a}$ which is non-free.

More generally, if $\Gamma$ contains some non-trivial normal subgroup $N$ for which the restriction $\bm{a}\resto N$ of $\bm{a}$ to $N$ is non-ergodic, then the action of $\Gamma$ on the set $Z$ of ergodic components of $\bm{a}\resto N$ corresponds to a non-trivial factor of $\bm{a}$ which is manifestly non-free.  Indeed, this factor is not even faithful as $N$ acts trivially on $Z$.

Working from the other direction, if $\pi : (X,\mu ) \ra (Y,\nu )$ factors $\bm{a}$ onto some non-trivial action $\bm{b} = \Gamma \cc ^b (Y,\nu )$ which is not faithful, then for any $B\subseteq Y$ with $0<\nu (B)<1$ the set $\pi ^{-1}(B)$ will be a non-trivial subset of $X$ witnessing that the kernel of $\bm{b}$ (i.e., the set of group elements fixing almost every point) does not act ergodically under the action $\bm{a}$.  These observations are rephrased in the following proposition.

\begin{proposition}\label{prop:normfree}
The following are equivalent for a measure preserving action $\bm{a}$ of $\Gamma$:
\begin{enumerate}
\item All non-trivial factors of $\bm{a}$ are faithful.
\item All non-trivial normal subgroups of $\Gamma$ act ergodically.
\end{enumerate}
\end{proposition}

Note that when $\Gamma$ contains a finite normal subgroup $N$ then no non-trivial action $\bm{a}= \Gamma \cc ^a (X,\mu )$ of $\Gamma$ can have the property (2) (and therefore (1)) of Proposition \ref{prop:normfree}: if $\bm{a} \resto N$ is ergodic then $X$ is finite, so the kernel of $\bm{a}$ is non-trivial and does not act ergodically. However, the observations preceding Proposition \ref{prop:normfree} also show the following:

\begin{proposition}\label{prop:normfree2}
The following are equivalent for a measure preserving action $\bm{a}$ of $\Gamma$:
\begin{enumerate}
\item All non-trivial factors of $\bm{a}$ have finite kernel.
\item All infinite normal subgroups of $\Gamma$ act ergodically.
\end{enumerate}
\end{proposition}

Propositions \ref{prop:normfree} and \ref{prop:normfree2} express the equivalence of a freeness property on the one hand, and an ergodicity property on the other. In this note we show that by strengthening the ergodicity assumption on $\bm{a}$, an appropriately strong freeness results.

\begin{definition}
A measure preserving action $\bm{a}$ of $\Gamma$ is called \emph{totally ergodic} if the restriction of $\bm{a}$ to every infinite subgroup of $\Gamma$ is ergodic.
\end{definition}

There are many examples of totally ergodic actions. All mildly mixing actions are totally ergodic for example, since the restriction of a mildly mixing action to an infinite subgroup is again mildly mixing and hence ergodic. In particular, all mixing actions are totally ergodic. The following theorem says that totally ergodic actions are, up to a finite kernel, always free.

\begin{theorem}\label{thm:ifree}
Let $\bm{a} = \Gamma \cc ^a (X,\mu )$ be a non-trivial measure preserving action of the countably infinite group $\Gamma$. Suppose that $\bm{a}$ is totally ergodic. Then there exists a finite normal subgroup $N$ of $\Gamma$ such that the stabilizer of $\mu$-almost every $x\in X$ is equal to $N$.
\end{theorem}

\begin{corollary}\label{thm:allmix}
All faithful totally ergodic actions of countably infinite groups are free. In particular, all faithful mildly mixing and all faithful mixing actions of countably infinite groups are free.
\end{corollary}

A totally ergodic action of particular importance is the \emph{Bernoulli shift} of $\Gamma$. This is the measure preserving action $\bm{s}_\Gamma$ of $\Gamma$ on $([0,1]^\Gamma ,\lambda ^\Gamma )$ (where $\lambda$ is Lebesgue measure) given by
\[
(\gamma ^{s_\Gamma}f)(\delta ) = f(\gamma ^{-1}\delta )
\]
for $\gamma ,\delta \in \Gamma$ and $f\in [0,1]^\Gamma$.  By a \emph{Bernoulli factor} of $\Gamma$ we mean a factor of $\bm{s}_\Gamma$. One consequence of Theorem \ref{thm:ifree} is a particularly nice group theoretic characterization of groups all of whose non-trivial Bernoulli factors are free.

\begin{corollary}\label{cor:tfae}
Let $\Gamma$ be an infinite countable group. Then the following are equivalent
\begin{enumerate}
\item Every non-trivial totally ergodic action of $\Gamma$ is free.
\item Every non-trivial mixing action of $\Gamma$ is free.
\item Every non-trivial Bernoulli factor of $\Gamma$ is free.
\item There exists a non-trivial measure preserving action $\bm{a}$ of $\Gamma$ such that every non-trivial factor of $\bm{a}$ is free.
\item There exists a non-trivial measure preserving action $\bm{a}$ of $\Gamma$ such that every non-trivial factor of $\bm{a}$ is faithful.
\item $\Gamma$ contains no non-trivial finite normal subgroup.
\end{enumerate}
\end{corollary}

\begin{proof}[Proof of Corollary \ref{cor:tfae} from Theorem \ref{thm:ifree}] (6)$\Ra$(1) follows immediately from Theorem \ref{thm:ifree}. The implication (1)$\Ra$(2) is clear. (2)$\Ra$(3) holds since $\bm{s}_\Gamma$ is mixing and every factor of a mixing action is mixing. (3)$\Ra$(4) and (4)$\Ra$(5) are also clear. (5)$\Ra$(6) follows from the discussion following Proposition \ref{prop:normfree} above.
\end{proof}

\begin{corollary}
Let $\Gamma$ be any infinite countable group that is either torsion free or ICC. Then every non-trivial totally ergodic action of $\Gamma$ is free and in particular every non-trivial Bernoulli factor of $\Gamma$ is free.
\end{corollary}

{\bf Acknowledgements.} I would like to thank Alekos Kechris for many useful comments and suggestions. I would also like to thank Benjy Weiss for his comments, and particularly for suggesting the term "total ergodicity." The research of the author was partially supported by NSF Grant DMS-0968710.

\section{Preliminary Definitions and Notation}

$\Gamma$ will always denote a countably infinite discrete group and $e$ will denote the identity element of $\Gamma$.

Let $\bm{a} = \Gamma \cc ^a (X,\mu )$ be a measure preserving action of $\Gamma$. The \emph{stabilizer} of a point $x\in X$ is the subgroup $\Gamma _x$ of $\Gamma$ given by
\[
\Gamma _x = \{ \gamma \in \Gamma \csuchthat \gamma ^a x = x \} .
\]
For a subset $C\subseteq \Gamma$ we let
\[
\mbox{Fix}^a(C) = \{ x\in X \csuchthat \forall \gamma \in C \ \, \gamma ^ax =x \} .
\]
We write $\mbox{Fix} ^a(\gamma )$ for $\mbox{Fix} ^a(\{ \gamma \} )$. The \emph{kernel} of $\bm{a}$ is the set
\[
\mbox{ker}(\bm{a}) = \{ \gamma \in \Gamma \csuchthat \mu (\mbox{Fix}^a(\gamma ) ) = 1 \} .
\]
It is clear that $\mbox{ker}(\bm{a})$ is a normal subgroup of $\Gamma$.

The action $\bm{a}$ is called \emph{(essentially) free} if the stabilizer of $\mu$-almost every point is trivial, or equivalently, $\mu (\mbox{Fix}^a(\gamma )) = 0$ for each $\gamma \in \Gamma \setminus \{ e \}$. It is called \emph{faithful} if $\mbox{ker}(\bm{a}) = \{ e \}$, i.e., $\mu (\mbox{Fix}^a(\gamma ))<1$ for each $\gamma \in \Gamma \setminus \{ e\}$.

Let $\bm{b} = \Gamma \cc ^b (Y,\nu )$ be another measure preserving action of $\Gamma$. We say that $\bm{b}$ is a \emph{factor} of $\bm{a}$ (or that $\bm{a}$ \emph{factors onto} $\bm{b}$) 
if there exists a measurable map $\pi : X \ra Y$ with $\pi _*\mu = \nu$ and such that for each $\gamma \in \Gamma$ the equality $\pi (\gamma ^a x )= \gamma ^b\pi (x)$ holds for $\mu$-almost every $x\in X$. 

\begin{definition}
A measure preserving action $\bm{b}= \Gamma \cc ^b (Y, \nu )$ is called \emph{trivial} if $\nu$ is a point mass. Otherwise, $\bm{b}$ is called \emph{non-trivial}.
\end{definition}

\section{Proof of Theorem \ref{thm:ifree}}

\begin{proof}[Proof of Theorem \ref{thm:ifree}]
We begin with a lemma also observed by Darren Creutz and Jesse Peterson \cite{CP12}.

\begin{lemma}\label{lem1}
Let $\bm{b} = \Gamma \cc ^b (Y,\nu )$ be a non-trivial totally ergodic action of $\Gamma$.
\begin{enumerate}
\item[(i)] Suppose that $C \subseteq \Gamma$ is a subset of $\Gamma$ such that $\nu ( \{ y \in Y\csuchthat C\subseteq \Gamma _y \} ) >0$. Then the subgroup $\langle C \rangle$ generated by $C$ is finite.
\item[(ii)] $\Gamma _y$ is almost surely locally finite.
\end{enumerate}
\end{lemma}

\begin{proof}[Proof of Lemma \ref{lem1}]
Beginning with (i), the hypothesis tells us that the set $\mbox{Fix}^b(C)$ is non-null. Since $\nu$ is not a point mass there is some $B\subseteq \mbox{Fix}^b(C)$ with $0<\nu (B) <1$. The set $B$ witnesses that $\bm{b}\resto \langle C \rangle$ is not ergodic. As $\bm{b}$ is totally ergodic we conclude that $\langle C\rangle$ is finite.

For (ii), let $\mc{F}$ denote the collection of finite subsets $F$ of $\Gamma$ such that $\langle F \rangle$ is infinite and let $\mbox{NLF}\subseteq Y$ denote the set of points $y\in Y$ such that $\Gamma _y$ is not locally finite. Then
\[
\mbox{NLF} = \bigcup _{F\in \mc{F}} \{ y\in Y \csuchthat F\subseteq \Gamma _y \}
\]
By part (i), $\nu ( \{ y\in Y\csuchthat F\subseteq \Gamma _y \} ) =0$ for each $F\in \mc{F}$. Since $\mc{F}$ is countable it follows that $\nu (\mbox{NLF} ) =0$.
\qedhere[Lemma]
\end{proof}

Now let $\bm{a}=\Gamma \cc ^a (X,\mu )$ be a totally ergodic action as in the statement of Theorem \ref{thm:ifree}. Let $N= \{ \gamma \in \Gamma \csuchthat \mu (\mbox{Fix}^a(\gamma )) =1 \}$ denote the kernel of $\bm{a}$. Then $N$ is a normal subgroup of $\Gamma$ that is finite by Lemma \ref{lem1}.(i). Ignoring a null set, the action $\bm{a}$ descends to an action $\bm{b} = \Delta \cc ^b (X,\mu )$ of the quotient group $\Delta = \Gamma /N$ that is still totally ergodic, and which is moreover faithful. Thus, after replacing $\Gamma$ by $\Gamma /N$ and $\bm{a}$ by $\bm{b}$ if necessary, we may assume that $\bm{a}$ is faithful toward the goal of showing that $\bm{a}$ is free.

For each $\gamma \in \Gamma$ let $C_\Gamma (\gamma )$ denote the centralizer of $\gamma$ in $\Gamma$. Observe that $\mbox{Fix}^a(\gamma )$ is an invariant set for $\bm{a}\resto C_\Gamma (\gamma )$, for if $\delta \in C_\Gamma (\gamma )$ then $\delta ^a \cdot \mbox{Fix}^a(\gamma ) = \mbox{Fix}^a(\delta \gamma \delta ^{-1}) = \mbox{Fix}^a(\gamma )$. Thus for $\gamma \neq e$, if $C_\Gamma (\gamma )$ is infinite then $\bm{a}\resto C_\Gamma (\gamma )$ is ergodic and the $\bm{a}\resto C_\Gamma (\gamma )$-invariant set $\mbox{Fix}^a(\gamma )$ must therefore be null since $\bm{a}$ is faithful. Letting $C_\infty$ denote the collection of elements of $\Gamma \setminus \{ e \}$ whose centralizers are infinite, this simply means that $\mu ( \{ x\in X \csuchthat \gamma \in \Gamma _x \} ) =0$ for all $\gamma \in C_\infty$, and so
\begin{equation}\label{eqn2}
\mu (\{ x\in X \csuchthat \Gamma _x \cap C_\infty \neq \emptyset \} ) =0 .
\end{equation}
By Lemma \ref{lem1}.(ii), $\Gamma _x$ is almost surely locally finite. By a theorem of Hall and Kulatilaka \cite{HK64} and Kargapolov \cite{K63}, every infinite locally finite group contains an infinite abelian subgroup. In particular, each infinite locally finite subgroup of $\Gamma$ intersects $C_\infty$. It follows from this and (\ref{eqn2}) that $\Gamma _x$ is almost surely finite.

Since there are only countably many finite subgroups of $\Gamma$ there must be some finite subgroup $H_0 \leq \Gamma$ such that $\mu (A_0) > 0$ where
\[
A_0 = \{ x\in X\csuchthat \Gamma _x = H_0 \} .
\]
Let $N_0$ denote the normalizer of $H_0$ in $\Gamma$. If $T$ is a transversal for the left cosets of $N_0$ in $\Gamma$ then $\{ t ^a A_0 \} _{t\in T}$ are pairwise disjoint non-null subsets of $X$ all of the same measure. It follows that $T$ is finite and therefore $N_0$ is infinite and $\bm{a}\resto N_0$ ergodic. The set $A_0$ is non-null and invariant for $\bm{a}\resto N_0$, so $\mu (A_0) =1$, i.e., $\Gamma _x = H_0$ almost surely. As $\bm{a}$ is faithful we conclude that $H_0 = \{ e \}$ and that $\bm{a}$ is therefore free.
\end{proof}

\section{An Example}

In general, total ergodicity does not imply weak mixing, and weak mixing does not imply total ergodicity. For example, the action of $\Z$ corresponding to an irrational rotation of $\T = \R /\Z$ equipped with Haar measure is totally ergodic, but not weakly mixing. There are also many examples of weakly mixing measure preserving actions that lack total ergodicity. One such action is exhibited in \ref{ex1} below. Example \ref{ex1} also illustrates that total ergodicity of a measure preserving action is \emph{not} necessary to ensure that each non-trivial factor of that action is free.

\begin{example}\label{ex1}
Here is an example of a free weakly mixing action $\bm{s}$ that is not totally ergodic, but that still has the property that every non-trivial factor of $\bm{s}$ is free: Let $F$ denote the free group of rank 2 with free generating set $\{ u, v \}$ and let $H= \langle u \rangle$ be the cyclic subgroup of $F$ generated by $u$. The generalized Bernoulli shift action $\bm{s}=\bm{s}_{F,F/H} = F\cc ^s ([0,1]^{F/H}, \lambda ^{F/H})$ is weakly mixing (see \cite{KT08}) but not totally ergodic since $H$ fixes each set in the $\sigma$-algebra generated by the projection function $f\mapsto f(H)$. Given a subgroup $K\leq F$, if $\bm{s}\resto K \cong \bm{s}_{K, F /H}$ is not ergodic then $K\cc F /H$ has a finite orbit (see \cite{KT08}), say $K\gamma H$ is finite where $\gamma \in F$. Then for any $z\in K$ there is some $n>0$ such that $z^n \in \gamma H\gamma ^{-1}$, and therefore $z\in \gamma H \gamma ^{-1}$. This shows that $K\subseteq \gamma H \gamma ^{-1}$ so that $K$ is cyclic.  The restriction of $\bm{s}$ to each non-cyclic subgroup of $F$ is therefore ergodic, so if $\bm{a} = F\cc ^a (X,\mu )$ is any factor of $\bm{s}$ then $\bm{a}$ also has this property and, assuming $\bm{a}$ is non-trivial, an argument as in the proof of Lemma \ref{lem1} shows that the stabilizer $F_x$ of $\mu$-almost every $x\in X$ is locally cyclic, hence cyclic. Arguing as in the last paragraph of the proof of Theorem \ref{thm:ifree} we see that there is some normal cyclic subgroup $H_0$ of $F$ such that $F _x = H_0$ almost surely. The only possibility is that $H_0= \{ e \}$, and thus $\bm{a}$ is free.
\end{example}

\section{A Question}

The proof of Theorem \ref{thm:ifree} relies on Hall, Kulatilaka, and Kargapolov's result, whose only known proofs make use of the Feit-Thompson theorem from finite group theory.

\begin{question}
Is there a direct ergodic theoretic proof of Theorem \ref{thm:ifree}?
\end{question}

$\ $\\
\noindent Department of Mathematics \\
California Institute of Technology \\
Pasadena, CA 91125 \\
\texttt{rtuckerd@caltech.edu}

\end{document}